\documentclass[12pt, reqno]{amsart}

\usepackage[usenames,dvipsnames]{xcolor}
\usepackage{amssymb, amsmath, amsthm,amsfonts}
\usepackage{tikz}
\usepackage{amscd}  
\usepackage{fullpage}
\usepackage{stmaryrd}
\usepackage{mathrsfs}
\usepackage[all]{xy} 
\usepackage{multirow} 
\usepackage{array}
\usepackage{braket}
\usepackage{setspace}

\newtheorem{lemma}{Lemma}[section]

\newtheorem{prop}[lemma]{Proposition}

\newtheorem{claim*}{Claim}
\newtheorem{thm}[lemma]{Theorem}

\theoremstyle{definition}

\newtheorem{rmk}[lemma]{Remark}

\newtheorem{example}[lemma]{Example}

\numberwithin{equation}{section}
\numberwithin{table}{section}

\newcommand{\defi}[1]{\textsf{#1}} 

\DeclareMathOperator{\known}{known}
\DeclareMathOperator{\tors}{tors}

\newcommand{\Q}{\mathbb{Q}}

\newcommand{\Z}{\mathbb{Z}}
\newcommand{\F}{\mathbb{F}}

\newcommand{\Qp}{\mathbb{Q}_{p}}
\newcommand{\Zp}{\mathbb{Z}_{p}}
\newcommand{\Fp}{\mathbb{F}_{p}}
\newcommand{\Cp}{\mathbb{C}_{p}}

\newcommand{\ord}{\text{ord}_p}

\newcommand{\ol}[1]{\overline{#1}}

\newcommand{\OmegaJ}{H^0(J_{\Qp}, \Omega^1)}
\newcommand{\OmegaC}{H^0(C_{\Qp}, \Omega^1)}
\newcommand{\Ann}{\operatorname{Ann}}
\newcommand{\rk}{\operatorname{rank}}
\renewcommand{\L}{\mathcal{L}}
\renewcommand{\div}{\operatorname{div}}

\title{Chabauty--Coleman experiments for genus 3 hyperelliptic curves}
\author{Jennifer S. Balakrishnan}
\author{Francesca Bianchi}
\author{Victoria Cantoral-Farf\'an} 
\author{Mirela \c{C}iperiani}
\author{Anastassia Etropolski}

\address{Jennifer S. Balakrishnan, Department of Mathematics and Statistics, Boston University, 111 Cummington Mall, Boston, MA 02215, USA}
\email{jbala@bu.edu}

\address{Francesca Bianchi, Mathematical Institute, University of Oxford, Andrew Wiles Building, Radcliffe Observatory Quarter, Woodstock Road, Oxford OX2 6GG, UK}
\email{francesca.bianchi@maths.ox.ac.uk }

\address{Victoria Cantoral-Farf\'an, ICTP - 11 Strada Costeira, Trieste, Italy}
\email{vcantora@ictp.it}

\address{Mirela \c{C}iperiani, Department of Mathematics, The University of Texas at Austin, 1 University Station, C1200 
Austin, Texas 78712, USA}
\email{mirela@math.utexas.edu}

\address{Anastassia Etropolski, Department of Mathematics, Rice University MS 136, Houston, TX 77251, USA }
\email{aetropolski@rice.edu}

\PassOptionsToPackage{hyphens}{url}
\usepackage[colorlinks, backref, breaklinks]{hyperref}

\usepackage[T1]{fontenc}

\begin{document}

\date{\today}
\maketitle

\begin{abstract}
We describe a computation of rational points on genus $3$ hyperelliptic curves $C$ defined over $\Q$ whose Jacobians have Mordell--Weil rank $1$. Using the method of Chabauty and Coleman, we present and implement an algorithm in \texttt{Sage} to compute the zero locus of two Coleman integrals and analyze the finite set of points cut out by the vanishing of these integrals. We run the algorithm on approximately 17,000 curves from a forthcoming database of genus 3 hyperelliptic curves and discuss some interesting examples where the zero set includes global points not found in $C(\Q)$. 
\end{abstract}

\section{Introduction}
Let $C$ be a non-singular curve over $\Q$ (or more generally, a number field $K$) of genus $g$. In the case where $g = 0$ or $g = 1$, $C$ has extra structure given by the fact that if $C(\Q)$ is non-empty, then $C$ is rational (if $g = 0$) or $C$ is an elliptic curve (if $g = 1$). In these cases, computing the set of rational points is either trivial by the Hasse Principle, or highly non-trivial in the case of elliptic curves. In the latter case the rational points form a finitely generated abelian group, and methods specific to this case exist for computing  upper bounds on the rank of $C(\Q)$, and the possibilities for the torsion subgroup of $C(\Q)$ are completely understood by the work of Mazur \cite[Theorem 8]{Mazur}.

On the other hand, if $g \ge 2$, then $C$ is of general type, and the Mordell conjecture, proved by Faltings in 1983 \cite{faltings}, implies that $C(\Q)$ is finite. Our main motivation is to compute $C(\Q)$ explicitly in this case. We will focus our attention on hyperelliptic curves of genus 3 such that the group of rational points of the Jacobian of $C$ has Mordell--Weil rank $r = 1$. This falls into the special case where $r < g$ which was considered by Chabauty in 1941 \cite{chabauty}, and techniques developed by Coleman in the 1980s allow us to use $p$-adic integration to bound, and often, in practice, explicitly compute, the set of rational points \cite{coleman,coleman_eff}.

In addition to these methods, we will also use the algorithm of Balakrishnan, Bradshaw, and Kedlaya \cite{explicitcoleman} and its implementation in \texttt{Sage} \cite{sage} to explicitly compute the relevant Coleman integrals by computing analytic continuation of Frobenius on curves. Nonetheless, we note that the algorithms presented in this article (see Section \ref{algorithm}) have not been implemented previously by other authors or carried out on a large collection of curves. (See, however, \cite{bruin2008} for related work in genus 2.)  Our code is available at \cite{code}.

We consider the case of genus $3$ hyperelliptic curves for two reasons:
\begin{enumerate}
	\item When $g = 3$, we can impose the condition that $0 < r < g -1$, i.e.\ $r = 1$, which, by a dimension argument, makes the method more effective. Indeed, in this case, the set $C(\Q)$ is contained in the intersection of the zero sets of the integrals of two linearly independent regular $1$-forms on the base-change of $C$ to $\Q_p$, where $p$ is any odd prime of good reduction.
	\item When $g = 2$, the Jacobian of $C$ is a surface, and its geometry and arithmetic is better understood. In particular, methods developed by Cassels and Flynn have been implemented by Stoll in \verb+Magma+ to make the computations needed much more efficient. More precisely, in this case, one can simplify the algorithm further by working with the quotient of the Jacobian by $\braket{\pm 1}$, which is a quartic surface in $\mathbb{P}^3$, known as the Kummer surface. In order to make the search of rational points more effective, the Chabauty method can also be combined with the Mordell--Weil sieve, which uses information at different primes (see also \cite{BruinStoll}).
\end{enumerate}

We begin with an overview of the Chabauty--Coleman method and explicit Coleman integration in Section \ref{preliminaries}.  In Section \ref{algorithm}, we present an algorithm to find a finite set of $p$-adic points containing the rational points of a hyperelliptic curve $C/\Q$ of genus $3$, which admits an odd model, and whose Jacobian $J$ has rank $1$. We fix a prime $p$ and work under the assumption that we know a $\Q$-rational point whose image in the Jacobian has infinite order (here the embedding of $C$ into $J$ is via the base-point $\infty$).  Besides $\Q$-rational points, the output will include all points in $C(\Q_p)$ which are in the pre-image of the $p$-adic closure of $J(\Q)$ in $J(\Q_p)$.

We then proceed to run our code on a list of relevant curves taken from the forthcoming database of genus 3 hyperelliptic curves \cite{g3hyp}. Our list consists of 16,977 curves, and we separately do a point search in \verb+Magma+ to find all $\Q$-rational points whose $x$-coordinates with respect to a fixed integral affine model have naive height at most $10^5$ (cf. Section \ref{dataanalysis}). Our Chabauty--Coleman computations then show that there are no $\Q$-rational points of larger height on any of these curves. 

In some cases, our algorithm outputs points in $C(\Q_p)\setminus C(\Q)$. Besides $\Q_p$-rational (but non-$\Q$-rational) Weierstrass points, on $75$ curves we find that the local point is the localization of a point  $P\in C(K)$ where $K$ is a quadratic field in which the prime $p$ splits. In all these cases, we are able to explain why these points appear in the zero locus that we are studying. The following three scenarios occur, and we discuss representative examples of each in Section \ref{dataanalysis}:

\begin{itemize}
\item It may happen that $[P- \infty]$ is a torsion point in the Jacobian (see Example \ref{ex:1}). In this case, the integral of any $1$-form would vanish between $\infty$ and $P$. 

\item As in Example \ref{ex:2}, it may happen that some multiple of the image of $[P-\infty]$ in the Jacobian actually belongs to $J(\Q)$: the vanishing here follows by linearity in the endpoints of integration. 

\item The Jacobian $J$ may decompose over $\Q$ as a product of an elliptic curve and an abelian surface. Then if the subgroup $H$ generated by $J(\Q)$ and the point $[P-\infty]$ comes from the elliptic curve, the dimension of the $p$-adic closure of $H$ in $J(\Q_p)$ must be equal to $1$, even if $[P-\infty]$ is a point of infinite order (see Example \ref{ex:3}).
\end{itemize}

\subsection*{Acknowledgements}
The first author is supported in part by NSF grant DMS-1702196, the Clare Boothe Luce Professorship (Henry Luce Foundation), and Simons Foundation grant \#550023. The second author is supported by EPSRC and by Balliol College through a Balliol Dervorguilla scholarship. The third author was supported by a Conacyt fellowship. The fourth author is supported by NSF grant DMS-1352598.\\ 
This project began at ``WIN4: Women in Numbers 4,'' and we are grateful to the conference organizers for facilitating this collaboration. We further acknowledge the hospitality and support provided by the Banff International Research Station.  We thank the Simons Collaboration on Arithmetic Geometry, Number Theory, and Computation for providing computational resources, and we are grateful to Bjorn Poonen, Andrew Sutherland, and Raymond van Bommel for helpful conversations.

\section{The Chabauty--Coleman method and Coleman integration}\label{preliminaries}
In this section, we review the Chabauty--Coleman method, used to compute rational points in our main algorithm. For further details, see Section \ref{algorithm}. We also give a brief overview of explicit Coleman integration on hyperelliptic curves.

\subsection{Chabauty--Coleman method}\label{sec:CCmethod}

Let $C$ be a smooth, projective curve over the rationals of genus at least 2. By the work of Faltings \cite{faltings}, we know $C(\Q)$ to be finite, but Faltings' proof does not explicitly yield the set $C(\Q)$. 
However, before the work of Faltings, Chabauty considered the following set-up. Let $p$ be a prime and $P\in C(\Q_p)$. Consider the embedding 
\begin{align*}
\iota_P \colon & C \hookrightarrow J \\
& Q \mapsto [Q - P].
\end{align*}
Then let $\overline{J(\Q)}$ denote the $p$-adic closure of $J(\Q)$ and define
\[
C(\Q_p) \cap \overline{J(\Q)} := \iota_P (C(\Q_p)) \cap \overline{J(\Q)}. 
\]
Chabauty proved the following case of Mordell's conjecture:

\begin{thm}[\cite{chabauty}]\label{th:chabauty}
Let $C/\Q$ be a curve of genus $g\geq 2$ such that the Mordell--Weil rank of the Jacobian $J$ of $C$ over $\Q$ is less than $g$, and let $p$ be a prime. Then $C(\Q_p) \cap \overline{J(\Q)}$ is finite.
\end{thm} 

Chabauty's result was later re-interpreted and made effective by Coleman, who showed the following:

\begin{thm}[\cite{coleman_eff}]\label{th:coleman}
	Let $C$ be as above and suppose that $p$ is a prime of good reduction for $C$. If $p > 2g$, then
	$$\#C(\Q) \le \#C(\F_p) + 2g - 2.$$
\end{thm}

To obtain an explicit upper bound on the size of $C(\Q_p) \cap \overline{J(\Q)}$, and hence $C(\Q)$, Coleman used his theory of $p$-adic integration on curves to construct $p$-adic integrals of 1-forms on $J(\Qp)$ that vanish on $J(\Q)$ and restrict them to $C(\Qp)$. Here, we follow the exposition in \cite{wetherell} in defining the Coleman integral.

Let $\omega_J \in \OmegaJ$, and $\lambda_{\omega_J}$ be the unique homomorphism $\lambda_{\omega_J} \colon J(\Q_p) \to \Q_p$ such that $d(\lambda_{\omega_J}) = \omega_J$. Consider the map induced by $\iota _P$
\[
\iota^* \colon \OmegaJ \to \OmegaC.
\]
Observe that $\iota^*$ is an isomorphism of vector spaces which is independent of the choice of $P \in C(\Q_p)$ \cite[Proposition 2.2]{Milne}.  

Define $\omega := \iota^*(\omega_J)$ to be the corresponding differential on $C$. On the Jacobian we have the natural pairing
\begin{align*}
\lambda \colon \OmegaJ \times J(\Q_p) & \to \Q_p\\
(\omega_J, R) & \mapsto \lambda_{\omega_J}(R) := \int_0^R \omega_J.
\end{align*}
Note that since $\lambda_{\omega_J}$ is a homomorphism, it vanishes on $J(\Qp)_{\text{tors}}$. Now given $P, Q \in C(\Q_p)$ we define 
\[
\int_P^Q \omega := \int_0^{[Q - P]} \omega_J,
\] 
hence for a fixed point $P \in C(\Q_p)$ and $\omega \in \OmegaC$ we get a function $\lambda_{\omega,P} \colon C(\Q_p) \to \Q_p$ with 
\[
\lambda_{\omega, P}(Q) := \int_P^{Q} \omega = \int_0^{[Q - P]} \omega_J = \lambda_{\omega_J}([Q - P]).
\]

We now restrict to the case where $g = g(C) = 3$ and $r = \rk J(\Q) = 1$, in which case $g - r =2$. The exposition below can be generalized whenever $r < g$. Let
\[
\Ann(J(\Q)) := \left\{ \omega_J \in \OmegaJ : \lambda_{\omega_J}(R) = \int_0^R \omega_J = 0 \text{ for all } R \in J(\Q)\right\}.
\]
This is a 2-dimensional $\Q_p$-vector space; hence there exist two linearly independent differentials $\alpha_J, \beta_J  \in \OmegaJ$ such that 
\[
\lambda_{\alpha_J}(R) = \lambda_{\beta_J}(R) = 0 \indent  \text{for all } R \in J(\Q). 
\]

Let $D$ be a $\Q$-rational divisor on $C$ of degree $r$, and consider the map $\iota_D \colon C \to J$ such that $Q \mapsto [rQ - D]$. Define 
\[
\lambda_{\omega, D}(Q) := \lambda_{\omega_J} \circ \iota_D(Q) =  \lambda_{\omega_J}(rQ - D).
\]
Consider the set
\begin{equation}\begin{aligned}\label{eqn:ZeroSet}
Z & := \{Q \in C(\Q_p) : \lambda_{\omega, D}(Q) = 0 \text{ for all } \omega \in \Ann(J(\Q))\} \\
& = \ker(\lambda_{\alpha,D}) \cap \ker(\lambda_{\beta,D}).
\end{aligned}\end{equation}
While a priori we have defined $Z$ in terms of $D$, it is actually independent of the choice of $D$, and $C(\Q) \subseteq Z$ \cite[\S 1.6]{wetherell}.

The above discussion indicates how we would handle the case when our hyperelliptic curve has an even degree model. However, since we restrict our attention to hyperelliptic curves $C$ with an odd degree model, we are guaranteed a rational point $\infty \in C(\Q)$ and we use $D = \infty$. Hence, we have two $\Q_p$-valued functions $\lambda_{\alpha,\infty}, \lambda_{\beta,\infty}$ on $C(\Q_p)$ whose common zeros capture the rational points of $C$.

%%%%%%%%%%%%%%%%%%%%%%%%%%%%%%%%%%%%%%%%%%%%%%%%%%%%%%%%%%

\subsection{Computing Coleman integrals}
In order to compute $Z$, we need a way to evaluate $\int_P^Q \omega$ for an arbitrary $\omega \in \OmegaC$ and arbitrary $P, Q \in C(\Qp)$. Suppose that $p$ is a prime of good reduction for $C$ and let $\overline{C}$ be the reduction of $C$ modulo $p$, i.e.\ the special fiber of a minimal regular proper model of $C$ over $\Zp$. Then there exists a natural reduction map $C(\Qp) \to \overline{C}(\Fp)$. Define a \defi{residue disk} to be a fiber of the reduction map. To compute $\int_P^Q \omega$, we now consider two cases: either $P$ and $Q$ lie in the same residue disk, or they do not.

\subsubsection{Coleman integral within a residue disk} 
Let $P\in C(\Q_p)$. By a \defi{local coordinate} for $P$ we mean a rational function $t\in \Q_p(C)$ such that
\begin{enumerate}
	\item $t$ is a uniformizer at $P$; and
	\item the reduction of $t$ to a rational function for $\overline{C}$ is a uniformizer at $\overline{P}$.
\end{enumerate} 
Hence, a local coordinate $t$ at $P$ establishes a bijection
\begin{align*}
p\Z_p&\leftrightarrow \{Q\in C(\Q_p): \overline{Q}=\overline{P}\} \\
t&\leftrightarrow (x(t),y(t)),
\end{align*}
where $x(t),y(t)$ are Laurent series and $(x(0),y(0))=P$. 

Suppose now that $\omega\in H^0(C_{\Q_p},\Omega^1)$ is not identically zero modulo $p$. We will often use the fact that the expansion of $\omega$ in terms of $t$ has the form $w(t)dt$, for some $w(t)\in\Z_p[[t]]$ converging on the entire residue disk. Hence, if $\overline{Q} = \overline{P}$, we can compute $\int_P^Q \omega$ by formally integrating a power series in the local coordinate $t$ (\cite[Lemma 7.2]{wetherell}). Such definite integrals are referred to as \defi{tiny integrals}.

A local coordinate at a given point $P\in C(\Q_p)$ can be found using \cite[Algorithms 2-4]{JenniferNatoLecture}. In particular,
\begin{enumerate}
	\item If $y(\overline{P})\neq 0$, then $x(t)=t+x(P)$ and $y(t)$ is the unique solution to $y^2=f(x(t))$ such that $y(0)=y(P)$.
	\item If $y(\overline{P})= 0$ and $\overline{P}\neq \overline{\infty}$, then $y(t)=t+y(P)$ and $x(t)$ is the unique solution to $f(x)=y(t)^2$ such that $x(0)=x(P)$.
	\item If $\overline{P}=\overline{\infty}$, one first finds $x(t)=t^{-2}+O(1)$ by solving $f(x)=\frac{x^6}{t^2}$. Then $y(t)=\frac{x(t)^3}{t}$.
\end{enumerate}
In practice, in all three cases one can explicitly compute $(x(t),y(t))$ up to arbitrary $p$-adic and $t$-adic precision by Newton's method.

\subsubsection{Coleman integral between different residue disks} In our intended application of computing rational points, we will fix a basepoint as one endpoint of integration and consider the various Coleman integrals given by varying the other endpoint of integration over all residue disks. This makes it essential that the tiny integrals constructed in the previous section are consistent across the set of residue disks: in other words, we need a notion of analytic continuation between different residue disks. 

Coleman solved this problem by using Frobenius to write down a unique ``path'' between different residue disks and presented a theory of $p$-adic line integration on curves \cite{coleman} satisfying a number of natural properties, among them linearity in the integrand, additivity in endpoints, change of variables via rigid analytic maps (e.g., Frobenius), and the fundamental theorem of calculus. This was made algorithmic in \cite{explicitcoleman} for \emph{hyperelliptic} curves by solving a linear system induced by the action of Frobenius on Monsky-Washnitzer cohomology, with an implementation available in \texttt{Sage}. 

The upshot is that given two points $P, Q \in C(\Q_p)$, one can compute the definite Coleman integral from $P$ to $Q$ as $\int_P^Q \omega$ directly via \cite{explicitcoleman}, as well as the Coleman integral from $P$ to the residue disk of $Q$, by further computing a local coordinate $t_Q$ at $Q$ (such that $t_Q|_{t=0} = Q$), which gives: $$\int_P^{t_Q} \omega = \int_P^Q \omega + \int_Q^{t_Q} \omega = \int_P^Q \omega + \int_0^t \omega,$$ where $\int_P^Q \omega$ now plays the role of the constant of integration between different residue disks.

\section{The algorithm}\label{algorithm}
We now specialize to our case of interest, where $C$ is a genus $3$ hyperelliptic curve given by an odd degree model, i.e.,
\[
C \colon y^2 = F(x)
\] 
where $F(x) \in \Q[x]$ is monic of degree $7$. We will further assume that the Jacobian $J$ of $C$ has Mordell--Weil rank $1$ over $\Q$. Finally, we will assume that we have computed a point $P_0 \in C(\Q)$ with the property that $[P_0-\infty]$ is of infinite order in $J(\Q)$. (This last assumption is straightforward to remove.)

Fix an odd prime $p$ of good reduction for $C$, denote by $\overline{C}$ the base change of $C$ to $\Fp$ and let $C(\Q)_{\known}$ denote a list of known points in $C(\Q)$. Given this input, the algorithm in this section returns the set $Z$ of common zeros of $\lambda_{\alpha,\infty}$ and $\lambda_{\beta,\infty}$, as defined in Section \ref{sec:CCmethod}, excluding the known rational points $C(\Q)_{\known}$.

\subsection{Upper Bounds in Residue Disks} \label{sec:upper_bounds}
Define $\omega_i=(x^i/2y)dx$ for $i \in \{0,1,2\}$. These differentials form a basis for $\OmegaC$. Let $\alpha$ and $\beta$ be $1$-forms in $\OmegaC$ such that $\alpha_J$ and $\beta_J$ form a basis for $\Ann(J(\Q))$ and such that $\alpha$ and $\beta$ are not identically zero modulo $p$. We may assume that we are in one of the following two situations:
\begin{enumerate}
	\item $\alpha=\omega_0$ and $\beta$ is a $\Z_p$-linear combination of $\omega_1$ and $\omega_2$, or
	\item $\alpha$ is a $\Z_p$-linear combination of $\omega_0$ and $\omega_1$ and $\beta$ is a $\Z_p$-linear combination of $\omega_0$ and $\omega_2$.
\end{enumerate}

Let $f'(t)$ be the local expansion of $\alpha$ or $\beta$ in the residue disk of a point $\overline{Q}\in \overline{C}(\F_p)$. Ultimately we want to compute the zeros of a particular antiderivative $f(t)$ lying in $p\Zp$ up to a desired $p$-adic precision. In certain cases, we will be able to avoid this calculation by instead obtaining an upper bound for the number of zeros of $f(t)$ in $p\Zp$ which we know to be sharp. To do this, we use the theory of Newton polygons for $p$-adic power series (see, e.g., \cite[IV.4]{Koblitz}).

Given $f(t)\in \Q_p[[t]]$ such that $f'(t)\in \Z_p[[t]]$, let $\overline {f'}(t)= f'(t) \pmod{p}$ and define
\begin{equation} \label{zero-ord}
m_f := \mathrm{ord}_{t=0} \overline {f'}(t).
\end{equation}
The following result is \cite[Lemma 5.1]{McCallum-Poonen} and can be viewed as a corollary of the $p$-adic Weierstrass Preparation Theorem \cite[Ch. IV, \S 4, Theorem 14]{Koblitz}.
\begin{lemma}
	\label{UpperBound}
	Let $f(t)\in \Q_p[[t]]$ such that $f'(t)\in \Z_p[[t]]$. If $m_f < p-2$, then the number of roots of $f$ in $p\Z_p$ is less than or equal to $m_f+1$.
\end{lemma}

\begin{rmk}
	\label{remarkRiemannRoch}
	 Note that if  $p>2g=6$ and $f'(t)$ is the local expansion of a regular $1$-form, then the Riemann-Roch Theorem implies that $m_f\leq 4$ and hence the condition $m_f<p-2$ of Lemma \ref{UpperBound} is always satisfied (cf. \cite[Theorem 5.3]{McCallum-Poonen}).
\end{rmk}

The following lemmas give a refinement of this result for our particular choice of $f$. We refer to a point of $C$ or $\overline C$ as a \defi{Weierstrass point} if it is fixed by the hyperelliptic involution.

\begin{lemma}
\label{nonWeierstrassclasses}
Let $f'(t)$ be the local expansion of $\alpha$ or $\beta$ in the residue disk of a point $\overline{Q}\in \overline{C}(\F_p)$. If $\overline{Q}$ is non-Weierstrass, then 
\[
m_f \leq 
\begin{cases}  
1 & \text{if } x(\overline{Q})\neq 0,\\ 
2 & \text{else}.
\end{cases}
\]
Moreover, the minimum of the orders of vanishing of $\alpha$ and $\beta$ at $\overline{Q}$ is less than or equal to $1$ for all non-Weierstrass $\overline{Q}$.
\end{lemma}

\begin{proof}
 By construction, the differential $f'$ is a linear combination of two of the differentials $\omega_i=(x^i/2y)dx$, $i=0,1,2$, and $f'$ is non-trivial modulo $p$. The assumption that $\overline{Q}$ is non-Weierstrass implies that $t=x-x(Q)$ is a local coordinate, where $Q$ is any lift of $\overline{Q}$ to characteristic zero. Write $f' = ((Ax^i+Bx^j)/2y)dx$, where $A,B\in\Z_p$, $i,j\in\{0,1,2\}$, $i < j$. Then in local coordinates we have
\begin{equation*}
f'(t)=\frac{A(t+x(Q))^i+B(t+x(Q))^j}{2y(t)}dt
\end{equation*}
where $y(t)$ has no zeros or poles in the residue disk. Since $A(t+x(Q))^i+B(t+x(Q))^j$ is a polynomial of degree less than or equal to $2$ in $t$, the first part of the first claim is proved. Furthermore, the polynomial has a double root modulo $p$ at $\overline{Q}$ if and only if $A\equiv 0\bmod p$, $j=2$, $B\not\equiv 0$, and $x(\overline{Q})= 0$; i.e.,\ if and only if $f'\equiv\frac{x^2}{2y}dx$ (up to rescaling) and $x(\overline{Q})=0$. The last statement also follows, since by construction $\alpha$  is a linear combination of $\omega_0$ and $\omega_1$ (see the beginning of \S \ref{sec:upper_bounds}).
\end{proof}

\begin{prop}
Let $p$ be an odd prime greater than or equal to $5$ of good reduction for $C$. Let $\overline{Q}\in\overline{C}(\F_p)$ be a non-Weierstrass point. Then the set
\begin{equation*}
\left\{P\in C(\Q_p): \overline{P}=\overline{Q}\text{ and }\int_{\infty}^P\alpha=\int_{\infty}^{P}\beta = 0\right\}
\end{equation*}
has size less than or equal to $2$.
\end{prop}

\begin{proof}
Follows from Lemma \ref{UpperBound} and Lemma \ref{nonWeierstrassclasses}.
\end{proof}

\begin{lemma}
\label{inftyclass}
Let $f'(t)$ be the local expansion of $\alpha$ or $\beta$ in the residue disk of the point $\overline{\infty}\in \overline{C}(\F_p)$. Then $m_f\in\{0,2,4\}$. In particular, the minimum of the orders of vanishing of $\alpha$ and $\beta$ at $\overline{\infty}$ is less than or equal to $2$.
\end{lemma}

\begin{proof}
We may take $x(t)=t^{-2}+O(1)$, $y(t)=t^{-7}+O(t^{-5})$ (cf. \cite[Algorithm 4]{JenniferNatoLecture}). Then $x^idx/2y$ has a zero of order $4-2i$ at $t=0$.
\end{proof}

\begin{prop}
Let $p$ be an odd prime greater than or equal to $5$ of good reduction for $C$. Then the set
\begin{equation*}
\left\{P\in C(\Q_p): \overline{P}=\overline{\infty}\text{ and }\int_{\infty}^P\alpha=\int_{\infty}^{P}\beta = 0\right\}
\end{equation*}
has size less than or equal to $3$. In particular, there are at most two points different from the point at infinity and reducing to it modulo $p$ in the above set.
\end{prop}

\begin{lemma}
\label{Weierstrassclasses}
Let $f'(t)$ be the local expansion of $\alpha$ or $\beta$ in the residue disk of a point $\overline{Q}\in \overline{C}(\F_p)$ (with the notation of the algorithm). If $\overline{Q}$ is Weierstrass, then 
\[
m_f \in 
\begin{cases}  
\{0,2\} & \text{if } x(\overline{Q})\neq 0,\\ 
\{0,2,4\} & \text{else}. 
\end{cases}
\]
Moreover, the minimum of the orders of vanishing of $\alpha$ and $\beta$ at $\overline{Q}$ is less than or equal to $2$.
\end{lemma}

\begin{proof}
In this case we may take $y=t$ and solve for $x$ using $y^2=F(x)$. In particular, then $x(t)=x(Q)+\frac{t^2}{F'(x(Q))}+O(t^4) \mod p$ (cf. \cite[Algorithm 3]{JenniferNatoLecture}). Therefore, $dx/2y$ has no zero or pole at $t=0$, $x^i dx/2y$ has either no zero or pole or a zero of order $2i$ if $\overline{Q}=(0,0)$.
Now consider
\begin{equation*}
f'(t)=\left(A\left(x(Q)+\frac{t^2}{f'(x(Q))}+O(t^4)\right)^i+B\left(x(Q)+\frac{t^2}{f'(x(Q))}+O(t^4)\right)^j\right) u(t)dt,
\end{equation*}
where $u(t)$ is a unit power series and $A$ is non-zero modulo $p$. For any choice of $i,j\in\{0,1,2\}$, $i< j$, it can be verified that $m_f\in \{0,2\}$ by distinguishing between the cases $x(\overline{Q})=0$ or $x(\overline{Q})\neq 0$. If $A\equiv 0\mod p$ when $i=0$ or $1$ and $j=2$ then $m$ equals $4$ if $x(\overline{Q})=0$. However, by construction, $\alpha$ and $\beta$ cannot both be of this form.
\end{proof}

\begin{prop}
Let $p$ be an odd prime greater than or equal to $5$ of good reduction for $C$. Let $\overline{Q}\in\overline{C}(\F_p)$ be a finite Weierstrass point. Then the set
\begin{equation*}
\left\{P\in C(\Q_p): \overline{P}=\overline{Q}\text{ and }\int_{\infty}^P\alpha=\int_{\infty}^{P}\beta = 0\right\}
\end{equation*}
has size less than or equal to $3$.
\end{prop}

%%%%%%%%%%%%%%
\subsection{Roots of $p$-adic power series}

Let $f'(t)$ be the local expansion of $\alpha$ (resp., $\beta$) in a residue disk, and let $f(t)$ be an antiderivative of $f'(t)$ whose constant term is either zero or the Coleman integral of $\alpha$ (resp., $\beta$) between $\infty$ and a $\Q_p$-rational point on $C$. To provably determine the roots of $f(t)$ lying in a residue disk up to a desired $p$-adic precision, we need to do the following:
\begin{itemize}
\item make sure that we truncate at a $p$-adic precision $p^N$ that is able to detect all the roots (up to $O(p^n)$ where $n=N-k$, see Proposition \ref{wheretotruncate}); 
\item determine $M$ such that to compute a root up to $O(p^n)$, we only need to consider the power series up to $O(t^M)$ where the coefficient of $t^i$ is in $O(p^n)$ for all $i\geq M$ if the roots are simple and $f$ is suitably normalized (i.e. $f\in \Z_p[[t]] \setminus p \Z_p[[t]]$).
\end{itemize}
Write $f(t)=f_M(t)+O(t^{M})$ where $M$ is an integer greater than or equal to $m_f+2$ and $f_M(t)$ is a polynomial of degree less than or equal to $M-1$. 
Then $f(t)$ and $f_M(t)$ have the same number of roots in $\Cp$ of $p$-adic valuation greater than or equal to $1$, as can be deduced from the same considerations on the Newton polygon of $f(t)$ which imply Lemma \ref{UpperBound} (for more details, see the proof of \cite[Lemma 5.1]{McCallum-Poonen}).
We are interested in the zeros of $f(pt)$ in $\Z_p$. Note that 
\[
f(pt)-f_M(pt)\in O(p^n,t^M) \text{ for some } n\geq M.
\] 
Hence, $f(pt)$ and $f_M(pt)$ as polynomials in $\Z/p^n\Z$ have exactly the same zeros (including multiplicities). Furthermore, if a zero of $f(pt)$ (and $f_M(pt)$) modulo $p^n$ is simple, then it lifts to a root of $f(pt)$ in $\Z_p$ by an inductive application of Hensel's lemma.

To compute a suitable choice of $M$, we require two more lemmas.

\begin{lemma}
	\label{nonanomalous}
	Let $\omega_i=(x^i/2y)dx$ for some $i\in\{0,1,2\}$, and $\lambda_i=\int_{\infty}^{P_0}\omega_i$. If $\overline{[P_0-\infty]}\in \overline{J}(\F_p)$ has order prime to $p$, then $\ord(\lambda_i)\geq 1$. In particular,  this holds if $p$ is a prime of non-anomalous reduction for $J$. 
\end{lemma}

\begin{proof}
	Let $n$ be the order of the reduction of $[P_0-\infty]$ modulo $p$. Then $Q = n [P_0-\infty]\in J_1(\Q_p)$, the kernel of reduction at $p$, and we have $\int_{\infty}^{P_0}\omega_i=\frac{1}{n}\int_{0}^Q\omega_{J,i}$, where $\iota^*(\omega_{J,i})=\omega_i$. Now $\int_{0}^Q \omega_{J,i}$ can be computed by writing $\omega_{J,i}$ as a power series in $\Z_p [[z_1,z_2,z_3]]$ where $z_1,z_2,z_3$ is a local coordinate system for $J_1(\Q_p)$ around $0$, formally integrating and evaluating at $z_1(Q),z_2(Q),z_3(Q)$. 
\end{proof}

\begin{lemma}
\label{val}
We have $f(pt)=\sum_{i=0}^{\infty}b_it^i=\sum_{j=0}^{\infty}\frac{a_j p^{j+1}}{j+1}t^{j+1}+c$, where $c\in \Q_p$, $a_j\in\Z_p$ for all $j\geq 0$. Therefore for all $i\geq 1$, $\ord(b_i)\geq i-\ord(i)$. Furthermore if $p^2\nmid \#J(\F_p)$ then $c\in\Z_p$.
\end{lemma}
\begin{proof}
The first assertion is clear. For the latter, recall that $c$ is either $0$ or of the form $\int_{\infty}^{Q}\gamma$, for some $Q\in C(\Q_p)$ and $\gamma\in\{\alpha, \beta\}$. The proof is then similar to Lemma \ref{nonanomalous}.
\end{proof} 

By Lemma \ref{val}, we know that $f(pt)$ has coefficients in $\Zp$, except possibly when $p^2|\#J(\Fp)$. Let $k$ be the minimum of the valuations of the coefficients of $f(pt)$. Note that, since $f'(t) \mod p$ has order of vanishing equal to $m_f$, if $m_f<p-2$, it follows that the valuation of the coefficient of $t^{m_f+1}$ in $f(pt)$ is precisely $m_f+1$. Therefore $k \leq m_f+1$. Furthermore, for $i>m_f+1$, we have $\ord(b_i)\geq i-\ord(i)>i-(i-m_f-1) = m_f + 1$.

\begin{prop} \label{wheretotruncate}
Let $f(t)$ be an antiderivative of $\alpha$ or $\beta$, let $m_f<p-2$, and let $k$ be the minimal valuation of the coefficients of $f(pt)$. Fix an integer $N$ such that $m_f+2\leq N\leq p^p-p$. Let $ap^e$ be the smallest integer greater than or equal to $N$ with $p\nmid a$ and $e\geq 1$, and set
\[
M= 
\begin{cases}  
ap^e+1 & \text{if } ap^e-e< N,\\ 
N & \text{else}. 
\end{cases}
\]
Then each simple root of $f_M(pt)$ in $\Z/p^{N-k}\Z$ equals the approximation modulo $p^{N-k}$ of a root of $f(pt)$. Furthermore, if all such roots are simple, then these are all the roots of $f(pt)$ in $\Z_p$. 
\end{prop}

\begin{proof}
It suffices to show that for $i\geq M$, $\ord(b_i)\geq N$. Since $M\geq N$, the statement is clear for $p\nmid i$ by Lemma \ref{val}. Now suppose $p|i$ for some $i\geq M$. Hence, $i= bp^r$ where $p\nmid b$ and $r\geq 1$, and $bp^r\geq M \geq N$.
Then by the definition of $ap^e$, we know that 
\[ 
bp^r \geq ap^e \indent \text{and} \indent 0\leq ap^e-N< p.     
\]
We now have two cases to consider:

\text{Case 1:} Assume that $bp^r= ap^e$. It follows that $M=N= bp^r$ which in turn implies that $ap^e-e\geq N$. Then since $(a,e)=(b,r)$, we have that 
\[
\ord(b_i)\geq bp^r-r\geq N.
\]  
\text{Case 2:}  Assume that $bp^r > ap^e$. It follows that $bp^r - ap^e\geq p$. Thus 
\[
\ord(b_i)\geq bp^r-r=(bp^r-ap^e)+(ap^e-r).
\] So if $\ord(b_i)<N$ then $r>(bp^r-ap^e)+(ap^e-N)\geq p$ and hence $p^p-p\leq bp^r-r<N$, contradicting our assumption on $N$.
\end{proof}

\begin{rmk}
In order to apply Proposition \ref{wheretotruncate} we need to meet the condition $m_f <p-2$; assuming that $p>2g$ guarantees  that this is always the case, as a consequence of the Riemann-Roch Theorem (see Remark \ref{remarkRiemannRoch}). Furthermore, in the case when $p>2g$, the hypothesis on $N$ of Proposition \ref{wheretotruncate} is always met, since $m_f+2<p<p^p-p$. 
\end{rmk}

%%%%%%%%%%%%%%%%%%%%%%%%%%%%%%%%%%%%%%%%
\subsection{Outline of the algorithm}\label{sec:algorithm}
We retain the notation of the beginning of Section \ref{algorithm}. The algorithm will always work if $p\geq 7$ and may or may not work if $p=3$ or $5$ (see Remark \ref{remarkRiemannRoch} and the comments in the main steps of the algorithm below). We now list the input and output of our algorithm followed by its main steps.

\smallskip
\text{Input:}
\begin {itemize}
\item $C$: a hyperelliptic curve of genus $3$ over $\Q$ given by a model $y^2=F(x)$ where $F\in \Q[x]$ is monic of degree $7$,  such that its Jacobian $J$ has rank $1$;
\item $p$: an odd prime of good reduction for $C$ not dividing the leading coefficient of $F$ and $p \geq 7$;
\item $P_0$: a point in $C(\Q)$ such that $[P_0-\infty] \in J(\Q)$ has infinite order;
\item $C(\Q)_{\known}$: a list of all known rational points on $C(\Q)$;
\item the $p$-adic precision $N$ (by Proposition \ref{wheretotruncate}, $N=2p+4$ is sufficiently large);
\item the $t$-adic precision $M$ (if $N=2p+4$ by Proposition \ref{wheretotruncate}, we can set $M= 2p+1$). 
\end{itemize}

\text{Output:}  The set $Z\subseteq C(\Q_p)$ defined in \eqref{eqn:ZeroSet} modulo the action of the hyperelliptic involution. In our code, this set is split into the following:
\begin {itemize}
\item a list of points of $Z$ which can be recognized as points in $C(\Q)\setminus C(\Q)_{\known}$ up to the hyperelliptic involution;
\item a list of points $P\in Z$ such that $[P-\infty]\in J(\Q_p)_{\tors}$, up to the hyperelliptic involution (here, if $P$ is not $2$-torsion and is the localization of a point defined over a quadratic extension of $K/\Q$ then the coordinates in $K$ are given as the corresponding minimal polynomials over $\Q$);
\item a list of all remaining points $P\in Z$ (as above, if $P$ is the localization of a point defined over a quadratic extension of $K/\Q$ then the coordinates in $K$ are given as the corresponding minimal polynomials over $\Q$).
\end{itemize}

\medskip
\text{Main steps of the algorithm:}

\begin{enumerate}
\item\label{alphabeta} \emph{A basis for the annihilator.} 

For each basis differential $\omega_i=(x^i/2y)dx$ ($i=0,1,2$), compute 
\[
\lambda_i=\int_{\infty}^{P_0}\omega_i \indent \text{modulo }p^n,
\]
where $n$ is the given $p$-adic precision. Set $k_{ij}:=\min\{\ord(\lambda_i),\ord(\lambda_j)\}$ and

\[
(\alpha, \beta)
 = 
\begin{cases}  
\left( \omega_0, \, p^{-k_{12}}(\lambda_1\omega_2-\lambda_2\omega_1) \right) & \text{if }\lambda_0=0,\\ 
\left( p^{-k_{01}}(\lambda_0\omega_1-\lambda_1\omega_0), \, p^{-k_{02}}(\lambda_0\omega_2-\lambda_2\omega_0) \right) & \text{else}.
\end{cases}
\]
In either case, $\alpha$ and $\beta$ are reductions modulo $p^{n'}$ of the pullback $\iota^*$ of a basis for the annihilator of $J(\Q)$, where 
\[
n'= 
\begin{cases}  
n-k_{12} & \text{if } \lambda_0=0,\\ 
n-\max\{k_{01},k_{02}\} & \text{else}. 
\end{cases}
\]
By Lemma \ref{nonanomalous}, $n'\leq n-1$ if $p$ is non-anomalous. If $n'\geq 6$ we are guaranteed to be able to carry out all computations in the next steps when $p\geq 7$.

\smallskip
\item\label{ruleout} \emph{Ruling out residue disks.} Observe that we only need to consider residue disks up to the hyperelliptic involution.

Reduce $\alpha$ and $\beta$ modulo $p$. For each $\overline{P}\in \overline{C}(\Fp)$, expand $\alpha$ and $\beta$ in a local coordinate $s$ around $\overline{P}$, calculate the orders of vanishing of $\alpha$ and $\beta$ at $s=0$, and let $m(\overline P)$ denote their minimum. Note that $m(\overline{P})\leq 2$ by Lemmas \ref{nonWeierstrassclasses}, \ref{inftyclass} and \ref{Weierstrassclasses}, and hence it suffices to compute $\alpha(s)$ and $\beta(s)$ up to $O(s^2)$ to find $m(\overline{P})$.

If $m(\overline{P})+1$ equals the number of $\Q$-rational points in $C(\Q)_{\known}$ reducing to $\overline{P}$ modulo $p$ and $m(\overline{P})<p-2$, then by Lemma \ref{UpperBound} the set $C(\Q)_{\known}$ contains all $\Q$-rational points in the residue disk of $\overline{P}$. Otherwise, proceed to the next step. 

\smallskip
\item\label{remainingclasses} \emph{Searching for the remaining disks.} 

If, for a given point $\overline{P}\in\overline{C}(\F_p)$, the number of $\Q$-rational points in $C(\Q)_{\known}$ reducing to $\overline{P}$ modulo $p$ is strictly smaller that $m(\overline{P})+1$, then we need to compute the set of $\Qp$-rational points $P$ reducing to $\overline{P}$ such that $\int_Q^{P}\alpha=\int_Q^{P}\beta = 0$ for a (any) rational point $Q$. For computational convenience we distinguish between two cases:
\begin{enumerate}
\item[(i)] If there exists $P\in C(\Q)_{\known}$ reducing to $\overline{P}$, let $t$ be a uniformizer at $P$. Then expand $\alpha$ and $\beta$ in $t$ and formally integrate to obtain two power series $f(t)$, $g(t)$, which parametrize the integrals of $\alpha$ and $\beta$ between $P$ and any other point in the residue disk.
\item[(ii)] If we do not know any $\Q$-rational point in the residue disk of $\overline{P}$, then we may assume that $\overline{P}\neq \overline{\infty}$ and hence write $\overline{P}=(\overline{x_0},\overline{y_0})$. If $\overline{y_0}=0$, let $P=(x_0,0)$ where $x_0$ is the Hensel lift of $\overline{x_0}$ to a root of $f(x)$. Otherwise, if $\overline{P}$ is not a Weierstrass point, we take $P=(x_0,y_0)$ where $x_0$ is any lift to $\Z_p$ of $\overline{x_0}$ (the Teichm\"uller lift of $\overline{x_0}$ would be a particularly convenient choice for $x_0$) and $y_0$ is obtained from $\overline{y_0}$ using Hensel's Lemma on $y^2=F(x_0)$. Let $\tilde{f}(t)$ and $\tilde{g}(t)$ be the integrals between $P$ and any other point reducing to $\overline{P}$ in terms of a local parameter $t$ at $P$. Then write $f(t)=\tilde{f}(t)+\int_{\infty}^P\alpha$ and $g(t)=\tilde{g}(t)+\int_{\infty}^P\beta.$
\end{enumerate} 
Recall that in (\ref{alphabeta}), we have computed the coefficients of the $\omega_i$ in $\alpha$ and $\beta$ modulo $n'$. To provably compute the set of common zeros to a desired precision, we require that one of $f$ or $g$ have only simple roots, except possibly at $t = 0$ (in practice, this has been the case for every curve that we have considered). To check this requirement, we compute their discriminants, which are correct up to the $p$-adic precision of the coefficients. 

Upon normalizing so that $t = 0$ is not a root of either $f$ or $g$, assume without loss of generality that $f$ has only simple roots. The $t$-adic precision we should compute $f(t)$ to in order to find provably correct approximations of its simple zeros is determined by Proposition \ref{wheretotruncate}. In practice, we truncate $f(t)$ at $O(t^M)$ where $M=n'$, unless the smallest multiple $r$ of $p$ greater than or equal to $n'$ satisfies $r-\ord(r)<n'$, in which case take $M=r+1$. For the $p$-adic precision, the coefficients are computed modulo $p^{n'}$. Then the simple roots are correct up to $O(p^{n'-k})$ where $k$ is the minimal valuation of the coefficients of $f(pt)$ (cf. the discussion preceding Proposition \ref{wheretotruncate}). To compute the roots we use the function \verb+polrootspadic+ implemented in \verb+PARI/GP+. Finally, we take the list of roots which lie in $p\Zp$ and check whether they are also roots of $g$.

If $p=3$ or $p=5$ and the order of vanishing of $f(t)$ or $g(t)$ modulo $p$ is greater than or equal to $p-2$, then we cannot provably find the zeros of $f(pt)$ and $g(pt)$. Currently the algorithm assumes that $p \ge 7$ to avoid these pitfalls.

\smallskip
\item \emph{Identifying the remaining classes.} 

Once we have found the common zeros of $f(pt)$ and $g(pt)$, we recover the corresponding $\Q_p$-rational points that do not come from points in $C(\Q)_{\known}$.
We now have the output set  that we will now break into sublists.

 If we fail to recognize a point $Q$ as $\Q$-rational, we can check whether the integral between $\infty$ and $Q$ of any non-zero differential $\gamma$ not in the span of $\alpha$ and $\beta$ also vanishes: if this is the case, the point $[Q-\infty]\in J(\Q_p)$ is torsion (cf. \cite[Proposition 3.1]{coleman}) and if we know explicitly $J(\Q)_{\text{tors}}$ (which in general is computable) we can verify whether $Q$ is $\Q$-rational or not. Furthermore, by increasing the degree in \texttt{algdep}, we may even try to  identify the number field over which the coordinates of $Q$ are defined\footnote{The hyperelliptic curve $C$ is defined over $\Q$. Thus the fact that $[Q-\infty]\in J(\Q_p)_{\text{tors}}$ forces $Q$ to have coordinates in $\overline{\Q}\cap \Q_p$.}. This may require high $p$-adic precision; however, it was possible for every curve we considered.
 
If the integral of the differential $\gamma$ is non-zero, and we have not recognized $Q$ as a $\Q$-rational point, we can still check whether the point $Q$ is defined over some number field $K$. For instance $[Q-\infty]$ could equal a point in $J(\Q)$ plus some torsion element in $J(K)$, with $[K:\Q]>1$ (see Example \ref{ex:1}).
\end{enumerate}

\subsection{Generalizations of the algorithm}

\subsubsection{What if we do not know $P_0\in C(\Q)$ such that $[P_0-\infty]\not\in J(\Q)_{\tors}$?\\}

The hyperelliptic curve we input in the algorithm is assumed to have rank $1$. Calculation of the rank is attempted by \verb+Magma+ \cite{magma} by working out both an upper bound and a lower bound, the former coming from computation of the $2$-Selmer group and the latter from an explicit search for linearly independent points on the Jacobian. The success of the rank computation relies on the two bounds being equal. In particular, if we suppose that we know provably that the rank of the Jacobian is one, we may as well assume that we know a point $Q \in J(\Q)$ of infinite order and a divisor $E$ on $C$ representing it. Then we may proceed as follows. The first task is to write $Q=[E]$ in the form $[D-d\infty]$, where $D$ is an effective $\Q$-rational divisor. In order to achieve this, we follow step by step the proof of \cite[Corollary 4.14]{stollsnotes}. That is, we compute the dimension of $\L(E+n\infty)$ for $n=0,1,2,\ldots$ (here $\L(E+n\infty)$ denotes the Riemann-Roch space of $E+n\infty$), until we find the smallest $n=m$ for which the dimension is $1$. Then $D-d\infty=E+\div(\phi)$, where $\phi$ generates $\L(E+m\infty)$. By \cite[Lemma 4.17]{stollsnotes}, $D$ is then the unique $\Q$-rational divisor in general position and of degree less than or equal to $g=3$ such that $Q$ can be represented in the form  $[D-d \infty]$.

Let $K$ be the smallest Galois extension of $\Q$ over which the support of $D$ is defined. Furthermore, let $p$ be a prime of good reduction for $C$ that splits completely in $K/\Q$ (there are infinitely many such primes). Then $K$ can be realized as a subfield of $\Q_p$ and hence the support of $D$ can be seen as lying in $C(\Q_p)$. Write $D=\sum_{i=1}^d P_i$ (some $P_i$ possibly being equal). Then we may proceed exactly as before, just replacing $\lambda_i$ in (\ref{alphabeta}) by
\begin{equation*}
\lambda_i=\sum_{i=1}^d\int_{\infty}^{P_i} \omega_i.
\end{equation*}

\subsubsection{Even degree model} 

The algorithm relies heavily on computations of Coleman integrals, for which one needs the hyperelliptic curve considered to have a model of the form $y^2=F(x)$, where $F(x)$ is monic. In particular, if one were to work with an even degree model, the two points at infinity would necessarily be defined over $\Q$ \cite{stollsnotes}. Therefore we could proceed as in the odd degree case with the single point at infinity being replaced by one of these two points. If $F(x)$ is not monic, the issue is that Coleman integration is not implemented in \verb+Sage+, though an implementation is available in \verb+Magma+ \cite{BT, colemangit}. Hence, general even models could be handled by computing the set of local points $Z$ as defined in \eqref{eqn:ZeroSet},  using the sum of the two points at infinity as the divisor $D$. 

\subsubsection{Other ranks and genera}

Our assumptions on $g$ and $r$ are somewhat arbitrary. With minor modifications, our code can be used in more general cases, provided that $0<r<g$. 

%%%%%%%%%%%%%%%%%%%%%%%%%%%%%%%%%%%%%%%%%
\section{Curve Analysis}\label{dataanalysis}

Once we implemented in \verb+Sage+ the algorithm described in the previous section, we ran it over $16,977$ hyperelliptic curves of genus $g=3$ satisfying the following properties:
\begin{enumerate}
\item the curve admits an odd degree model over $\Q$;
\item the Jacobian of the curve has Mordell--Weil rank equal to $1$;
\item there is a $\mathbb{Q}$-rational point $P_0$ such that $[P_0-\infty]$ has infinite order in $J(\Q)$.
\end{enumerate}

In order to obtain those curves, we sorted the $67,879$ genus $3$ hyperelliptic curves from a forthcoming database of genus $3$ curves over $\Q$ \cite{g3hyp}.  Out of these, $19,254$ curves satisfy conditions $(1)$ and $(2)$. 

Running our code for the $16,977$ curves for which we could further find a $P_0$ as in $(3)$, we found $75$ curves where the zero set $Z$ contains something other than the rational points we had already computed and Weierstrass points defined over $\Qp$ for our chosen prime $p$. Note that in all $16,977$ computations, the prime $p$ used was the smallest prime greater than $2g=6$ which divided neither the discriminant nor the leading coefficient of the hyperelliptic polynomial defining the curve.

Let $C$ be one of these $75$ curves, let $W$ denote the set of Weierstrass points in $C(\Qp)\setminus C(\Q)$ and let $P \in Z \setminus (C(\Q) \cup W)$. In all cases, we identified $P$ as a point defined over a quadratic extension $K$ of $\mathbb{Q}$ in which $p$ splits. Even so, these 75 curves split up into 3 distinct cases:
\begin{enumerate}
\item $[P - \infty] \in J(K)_{\text{tors}}$;
\item $[P-\infty]\not\in J(K)_{\text{tors}}$ but $n [P-\infty]\in J(\Q)$ for some positive integer $n$;
\item $\langle J(\Q), [P - \infty]\rangle_{\mathbb{Z}}$ is a rank $2$ subgroup of $J(K)$.
\end{enumerate} 

In cases $(1)$ and $(2)$, it is clear why $P\in Z$. On the other hand, justifying case $(3)$ requires investigating more closely the geometry of the Jacobian of the curve, as is carried out in detail in Example \ref{ex:3}.

\subsection{Examples} For each of the curves below we give a list of known rational points, which are all of the rational points up to a height\footnote{
Our computations show that the $\mathbb{Q}$-rational points of highest absolute logarithmic height (with respect to an odd degree model) on a curve among the $16,977$ hyperelliptic curves that we considered, are
\[
\left(-\frac{49}{18} , -\frac{339563}{11664}\right), \left(-\frac{49}{18},-\frac{1600445}{52488}\right)
\]
on the hyperelliptic curve
\[
C:y^2 + (x^4+x^2+x)y= x^7-x^6-5x^5+5x^3-3x^2-x, 
\]
which has absolute minimal discriminant $5326597$ and whose Jacobian has conductor $5326597$.
} of $10^5$. Following the algorithm outlined in \S \ref{sec:algorithm}, we produce the set $Z$ of local points for the prime $p = 7$ or $p=11$, which in each case returns no new $\Q$-rational points, hence concluding that the set of known rational points $C(\Q)_{\known}$ is all of $C(\Q)$. In each of the examples below, however, $Z$ contains a $\Qp$-point which is not a Weierstrass point and falls into one of the cases outlined above.

\begin{example}\label{ex:1} Consider the hyperelliptic curve 
\[
y^2 + x^3y = x^7+2x^6-2x^5-9x^4-4x^3+8x^2+8x+2
\] 
(given above by a minimal model) which has absolute minimal discriminant  $544256$ and whose Jacobian has conductor 
$544256=2^9 \cdot 1063$. We work with an odd degree model
\[
C: y^2 = 4x^7 + 9x^6 - 8x^5 - 36x^4 - 16x^3 + 32x^2 + 32x + 8, 
\]
which has the following five known rational points:
\[
C(\Q)_{\mathrm{known}}=\{\infty, (-1, -1), (-1, 1), (1, -5), (1, 5)\}.
\]
Running the code on $C$ together with the prime $p=7$ and the point $P_0=(-1,-1)$, we find that 
\[
Z=C(\mathbb{Q})_{\mathrm{known}}\cup W \cup \{(0,\pm 2\sqrt{2})\},
\] 
where the set $W$ of non-$\Q$-rational Weierstrass points has size $3$. Moreover, the points $[(0,\pm 2\sqrt{2})-\infty]\in J(\Q(\sqrt{2}))$ have order $12$. 
\end{example}

\medskip
\begin{example}\label{ex:2} 
Consider the hyperelliptic curve 
\[
y^2 + (x^4 + 1)y = 2x^3 + 2x^2 + x
\] 
(given above by a minimal model) which has absolute minimal discriminant $48519$ and whose Jacobian $J$ has conductor $48519=3^4 \cdot 599$. We work with an odd degree model
\[
C: y^2 = -4x^7 + 24x^6 - 56x^5 + 72x^4 - 56x^3 + 28x^2 - 8x + 1.\]
This curve has the following five known rational points:
\[
C(\Q)_{\mathrm{known}} = \{\infty, (0,-1), (0,1), (1,-1), (1,1)\}.
\]

For this example we will give a more detailed outline of the algorithm. Following Section \ref{sec:CCmethod} we know that there exist functions $\lambda_{\alpha,\infty}$, $\lambda_{\beta,\infty}$ on $C(\Q_p)$, corresponding to differentials $\alpha, \beta \in \OmegaC$. These two functions vanish on the rational points of $C$, i.e., 
\[
C(\Q) \subseteq \ker(\lambda_{\alpha,\infty}) \cap \ker(\lambda_{\beta,\infty}) = Z.
\] 
We would like to know whether this zero set $Z$ contains anything other than the $\Q$-rational points on $C$.

If we take $p=11$, we find that $Z = C(\Q)$. Observe that  $11$ is inert in $K = \Q(\sqrt{-3})$.

If we take $p=7$, which splits in $K = \Q(\sqrt{-3})$, we find that there are four points defined over $K = \Q(\sqrt{-3})$ that appear in $Z$. Up to hyperelliptic involution, we have
\[
\left\{\left((1+\sqrt{-3})/2, \sqrt{-3}\right), \left((1-\sqrt{-3})/2, \sqrt{-3}\right)\right\} \subseteq Z.
\]
There is a good reason for the presence of these points in $Z$: if $P$ denotes any of the above points, then $5[P - \infty] \in J(\Q)$, therefore $5\lambda_{\alpha}(P) = 0$.

We now run through the algorithm to see that for $p=7$ we find that
\[
C(\Q)_{\mathrm{known}} \cup \left\{\left((1+\sqrt{-3})/2, \sqrt{-3}\right), \left((1-\sqrt{-3})/2, \sqrt{-3}\right)\right\}=Z \indent \text{up to hyperelliptic involution}.
\]
First we change variables to obtain an equation for $C$ where the defining polynomial $F(x)$ is monic, so we send $x \mapsto -4x$, $y \mapsto 4^4 y$. The $\F_7$-points of $C$ are 
\[
C(\F_7) =\{\ol{\infty}, \ol{(0,2)}, \ol{(0,5)}, \ol{(1,4)}, \ol{(1,3)}, \ol{(2,4)}, \ol{(2,3)}, \ol{(4,4)}, \ol{(4,3)}, \ol{(5,2)}, \ol{(5,5)}\}.
\]
Of these eleven points, five of them arise as reductions of known $\Q$-rational points, and an order of vanishing calculation shows that these are the only rational points in those residue disks. For the remaining six $\F_7$-points of $C$, the same order of vanishing calculation shows that there is at most one $\Q_p$-point in each residue disk corresponding to these points on which $\lambda_{\alpha,\infty}$ and $\lambda_{\beta,\infty}$ vanish. 

We know that the four quadratic points above reduce to 
\[
\{\ol{(1,4)}, \ol{(1,3)}, \ol{(4,4)},\ol{(4,3)}\}
\]
in some order (note that the quadratic points listed above are on the original curve, and these $\F_7$-points are the images after the change of variables of their reductions).  Hence, our task is now reduced to the analysis of the residue disks of $\ol{(2,4)}$ and $\ol{(2,3)}$, which moreover map to each other under the hyperelliptic involution. To show that $\lambda_{\alpha}, \lambda_{\beta}$ have no zeros in these residue disks, we will explicitly write down the power series and compute their zeros using \verb+PARI/GP+, as outlined in \S\ref{sec:CCmethod}.

As usual, let $\omega_i = x^i/2y$. Then the annihilator of $J(\Q)$ under the integration pairing is spanned by
\begin{align*}
\alpha & = (1 + 2\cdot 7 + 7^2 + 2\cdot 7^3 + 5\cdot 7^4 + O(7^5))\,\omega_0 + (4 + 7^2 + 5\cdot 7^4 + O(7^5))\,\omega_1, \\
\beta & = (6 + 3\cdot 7 + 2\cdot 7^2 + 5\cdot 7^4 + O(7^5))\,\omega_0 + (4 + 7^2 + 5\cdot 7^4 + O(7^6))\,\omega_2.
\end{align*} 

It suffices to consider the residue disk of $\ol{(2,4)}$. In this residue disk, we obtain the two power series
\begin{align*}
f(t) & = 2\cdot7 + 4\cdot7^2 + O(7^3) + (2 + 7 + 2\cdot7^2 + O(7^3))t + (6 + 5\cdot7 + 4\cdot7^2 + O(7^3))t^2 + \cdots, \\
g(t) & = 7 + 6\cdot7^2 + O(7^3) + (1 + 5\cdot7 + 2\cdot7^2 + O(7^3))t + (3 + 4\cdot7 + 7^2 + O(7^3))t^2 + \cdots.
\end{align*}
Each of these have one zero in $p\Zp$, but not the same zero. The two zeros are $6\cdot7 + 5\cdot7^2 + O(7^3)$ and $6\cdot7 + 2\cdot7^2 + O(7^3)$, respectively.
\end{example}

%%%%%%%%%%%%%%%%%%%%%%%%%%
\begin{example}\label{ex:3} 
Consider the hyperelliptic curve 
\[
y^2 + (x^3+x)y = x^7-4x^6+8x^5-10x^4+8x^3-4x^2+x,
\] 
(given above by a minimal model) which has absolute minimal discriminant $1573040$ and whose Jacobian $J$ has conductor $786520 = 2^3\cdot 5 \cdot 7 \cdot 53^2$. We work with the odd degree model
\[
C:y^2=4x^7 - 15x^6 + 32x^5 - 38x^4 + 32x^3 - 15x^2 + 4x, 
\]
on which we know the following rational points:
\[
C(\mathbb{Q})_{\mathrm{known}} =\{\infty, (0, 0), (1, -2), (1, 2)\}.
\]
The point $R=[(1,-2)-\infty]$ has infinite order in $J(\mathbb{Q})$ and can thus be used to initiate the algorithm of \S\ref{sec:algorithm} with $p=11$, which is the smallest prime greater than $6$ of good reduction for $C$. We find that
\[
Z = C(\mathbb{Q})_{\mathrm{known}}\cup W\cup \{(-1,\pm 2\sqrt{-35})\},
\]
where $W$ has size $2$. In particular, we have $C(\mathbb{Q})=C(\mathbb{Q})_{\mathrm{known}}$. 

Perhaps more interestingly, we now explain\footnote{We are grateful to Andrew Sutherland for kindly computing real endomorphism algebras (using the techniques of \cite{g2data, HS}) for a number of curves that produced Chabauty--Coleman output similar to this example, which greatly assisted in understanding the structure of their Jacobians.  We would also like to thank Bjorn Poonen for a very helpful discussion about this phenomenon.} why $\{(-1,\pm 2\sqrt{-35})\}\subseteq Z$. Let $K=\mathbb{Q}(\sqrt{-35})$, $\mathrm{Gal}(K/\mathbb{Q})=\braket{\tau}$ and fix an embedding $K\to \mathbb{Q}_p$. The point 
\[
Q=[(-1,2\sqrt{-35})-\infty]\in J(K)
\] 
is of infinite order, as there exists a non-zero differential in $H^0(J_{\mathbb{Q}_p},\Omega^1)$ whose integral does not vanish on $Q$. Suppose that $nQ\in J(\mathbb{Q})$ for some integer $n\neq 0$. Then $nQ=(nQ)^{\tau}=nQ^{\tau}$ and hence $2Q = Q-Q^{\tau}\in J(K)_{\text{tors}}$, a contradiction. It follows that $J(\mathbb{Q})$ and $Q$ generate a subgroup of rank $2$ in $J(K)$. A computation in \verb+Magma+ shows that the rank of $J(K)$ itself is equal to $2$. Therefore, the image of $Z\setminus W$ in $J(K)$ generates a subgroup of finite index. Showing that $(-1,\pm 2\sqrt{-35})\in Z$ is then equivalent to proving that the dimension of the $p$-adic closure of $J(K)$ in $J(\mathbb{Q}_p)$ is $1$, i.e. that the $\mathbb{Z}$-linear independence of $Q$ and $J(\mathbb{Q})$ is not preserved under base-change to $\mathbb{Q}_p$.

To explain this phenomenon, we compute the automorphism group of $C$ using \verb+Magma+. We have that $\mathrm{Aut}(C)\cong C_2\times C_2$, where the first copy of $C_2$ is generated by the hyperelliptic involution $i:C\to C$ and the second one is generated by $\phi:C\to C$, $\varphi(x,y,z)=(z,y,x)$. The quotient $C/\braket{\phi}$ is the elliptic curve
\[
E:y^2 + xy + y = x^3 - x^2,
\]
whereas the quotient $C/\braket{\varphi\circ i}$ is the genus $2$ hyperelliptic curve
\[
H:y^2 + (x^2 + 1)y = x^5 + x^4 - 4x^3 + 3x^2 - x - 1.
\]
It follows that $J$ decomposes, over $\mathbb{Q}$, as a product of $E$ and the Jacobian of $H$, which is an abelian surface $A$ with no extra endomorphisms \cite[Theorem 4]{Paulhus} . 

Since $\text{rank}(E(\mathbb{Q}))=\text{rank}(J(\Q))=1$ and the $p$-adic closure of $E(L)$ in $J(\mathbb{Q}_p)$ can have dimension at most one for any number field $L$ where $p$ splits completely and any embedding $L\to\Q_p$, in order to determine which points of $C(\overline{\mathbb{Q}})$ can appear in $Z$, we then need to search for points that map to torsion points in $A$, under the quotient map $C\to C/\braket{\varphi\circ i}\to A$. 

An explicit computation using Coleman integrals on $H$ shows that $H(\Q_p)\cap A(\Q_p)_{\mathrm{tors}}\subseteq A(\Q_p)[2]$, where, as usual, the embedding of $H$ into $A$ is via the base-point $\infty$. Let $T=(x,y,z)\in C(\Q_p)$. Then $\varphi\circ i (T)=(z,-y,x)$: thus, $T$ maps into $A(\Q_p)_{\mathrm{tors}}$ iff either $T=(0,0)$, $T=\infty$ or $(z,-y,x)=(x,-y,z)$, i.e. $x^2=1$. This shows both why $\{(-1,\pm 2\sqrt{-35})\}\subset Z$ and why no other non-torsion point in $C(\overline{\Q})$ can occur in $Z$ besides $(1,\pm 2)$ and $(-1,\pm 2\sqrt{-35})$.
\end{example}

\bibliographystyle{amsalpha}
\bibliography{biblioarticle} 

\newcommand{\etalchar}[1]{$^{#1}$}
\providecommand{\bysame}{\leavevmode\hbox to3em{\hrulefill}\thinspace}
\providecommand{\MR}{\relax\ifhmode\unskip\space\fi MR }
% \MRhref is called by the amsart/book/proc definition of \MR.
\providecommand{\MRhref}[2]{%
  \href{http://www.ams.org/mathscinet-getitem?mr=#1}{#2}
}
\providecommand{\href}[2]{#2}
\begin{thebibliography}{BSS{\etalchar{+}}16}

\bibitem[Bal15]{JenniferNatoLecture}
J.~S. Balakrishnan, \emph{Explicit {$p$}-adic methods for elliptic and
  hyperelliptic curves}, Advances on superelliptic curves and their
  applications, NATO Sci. Peace Secur. Ser. D Inf. Commun. Secur., vol.~41,
  IOS, Amsterdam, 2015, pp.~260--285. \MR{3525580}

\bibitem[BBCF{\etalchar{+}}]{code}
J.~S. Balakrishnan, F.~Bianchi, V.~Cantoral-Farf\'an, M.~\c{C}iperiani, and
  A.~Etropolski, \emph{Sage code}, \url{https://github.com/jbalakrishnan/WIN4}.

\bibitem[BBK10]{explicitcoleman}
J.~S. Balakrishnan, R.~W. Bradshaw, and K.~S. Kedlaya, \emph{Explicit {C}oleman
  integration for hyperelliptic curves}, Algorithmic number theory, Lecture
  Notes in Comput. Sci., vol. 6197, Springer, Berlin, 2010, pp.~16--31.
  \MR{2721410}

\bibitem[BCP97]{magma}
W.~Bosma, J.~Cannon, and C.~Playoust, \emph{The {M}agma algebra system. {I}.
  {T}he user language}, J. Symbolic Comput. \textbf{24} (1997), no.~3-4,
  235--265, Computational algebra and number theory (London, 1993).
  \MR{MR1484478}

\bibitem[BPSS]{g3hyp}
A.~Booker, D.~Platt, J.~Sijsling, and A.~Sutherland, \emph{Genus 3
  hyperelliptic curves},
  \url{http://math.mit.edu/~drew/gce\_genus3\_hyperelliptic.txt}.

\bibitem[BS08]{bruin2008}
N.~Bruin and M.~Stoll, \emph{Deciding existence of rational points on curves:
  An experiment}, Experiment. Math. \textbf{17} (2008), no.~2, 181--189.

\bibitem[BS10]{BruinStoll}
\bysame, \emph{The {M}ordell-{W}eil sieve: proving non-existence of rational
  points on curves}, LMS J. Comput. Math. \textbf{13} (2010), 272--306.
  \MR{2685127}

\bibitem[BSS{\etalchar{+}}16]{g2data}
A.~R. Booker, J.~Sijsling, A.~V. Sutherland, J.~Voight, and D.~Yasaki, \emph{A
  database of genus-2 curves over the rational numbers}, LMS J. Comput. Math.
  \textbf{19} (2016), no.~suppl. A, 235--254.

\bibitem[BT]{colemangit}
J.~S. Balakrishnan and J.~Tuitman, \emph{Magma code},
  \url{https://github.com/jtuitman/Coleman}.

\bibitem[BT17]{BT}
\bysame, \emph{Explicit {C}oleman integration for curves}, Arxiv preprint
  (2017).

\bibitem[Cha41]{chabauty}
C.~Chabauty, \emph{Sur les points rationnels des courbes alg\'ebriques de genre
  sup\'erieur \`a l'unit\'e}, C.R. Acad. Sci. Paris \textbf{212} (1941),
  882--885.

\bibitem[Col85a]{coleman_eff}
R.~F. Coleman, \emph{Effective {C}habauty}, Duke Math. J. \textbf{52} (1985),
  no.~3, 765--770. \MR{808103}

\bibitem[Col85b]{coleman}
\bysame, \emph{Torsion points on curves and {$p$}-adic abelian integrals}, Ann.
  of Math. (2) \textbf{121} (1985), no.~1, 111--168. \MR{782557}

\bibitem[Fal83]{faltings}
G.~Faltings, \emph{Endlichkeitss{\"a}tze f{\"u}r abelsche {V}ariet{\"a}ten
  {\"u}ber {Z}ahlk{\"o}rpern}, Invent. Math. \textbf{73} (1983), no.~3,
  349--366.

\bibitem[HS16]{HS}
D.~Harvey and A.~V. Sutherland, \emph{Computing {H}asse-{W}itt matrices of
  hyperelliptic curves in average polynomial time, {II}}, Frobenius
  distributions: {L}ang-{T}rotter and {S}ato-{T}ate conjectures, Contemp.
  Math., vol. 663, Amer. Math. Soc., Providence, RI, 2016, pp.~127--147.

\bibitem[Kob84]{Koblitz}
N.~Koblitz, \emph{{$p$}-adic numbers, {$p$}-adic analysis, and zeta-functions},
  second ed., Graduate Texts in Mathematics, vol.~58, Springer-Verlag, New
  York, 1984. \MR{754003}

\bibitem[Maz77]{Mazur}
B.~Mazur, \emph{Modular curves and the \protect{Eisenstein} ideal}, Inst.
  Hautes \'Etudes Sci. Publ. Math. (1977), no.~47, 33--186 (1978).

\bibitem[Mil86]{Milne}
J.~S. Milne, \emph{Jacobian varieties}, Arithmetic geometry ({S}torrs, {C}onn.,
  1984), Springer, New York, 1986, pp.~167--212. \MR{861976}

\bibitem[MP12]{McCallum-Poonen}
W.~McCallum and B.~Poonen, \emph{The method of {C}habauty and {C}oleman},
  Explicit methods in number theory, Panor. Synth\`eses, vol.~36, Soc. Math.
  France, Paris, 2012, pp.~99--117.

\bibitem[Pau08]{Paulhus}
J.~Paulhus, \emph{Decomposing {J}acobians of curves with extra automorphisms},
  Acta Arithmetica \textbf{132} (2008), no.~3, 231--244 (eng).

\bibitem[S{\etalchar{+}}17]{sage}
W.~A. Stein et~al., \emph{{S}age {M}athematics {S}oftware ({V}ersion 8.1)}, The
  Sage Development Team, 2017, {\tt http://www.sagemath.org}.

\bibitem[Sto14]{stollsnotes}
M.~Stoll, \emph{Arithmetic of {H}yperelliptic {C}urves}, 2014.

\bibitem[Wet97]{wetherell}
J.~L. Wetherell, \emph{Bounding the number of rational points on certain curves
  of high rank}, ProQuest LLC, Ann Arbor, MI, 1997, Thesis (Ph.D.)--University
  of California, Berkeley. \MR{2696280}

\end{thebibliography}

\end{document}